\documentclass[12pt, reqno]{amsart}
\usepackage{amsmath, amsthm, amscd, amsfonts, amssymb, graphicx, color, mathrsfs}
\usepackage[bookmarksnumbered, colorlinks, plainpages]{hyperref}
\usepackage[all]{xy}
\usepackage{slashed}

\usepackage{soul}
\usepackage{cancel}
\usepackage{ulem}

\newtheorem{theorem}{Theorem}[section]
\newtheorem{lemma}[theorem]{Lemma}
\newtheorem{proposition}[theorem]{Proposition}
\newtheorem{corollary}[theorem]{Corollary}
\theoremstyle{definition}
\newtheorem{definition}[theorem]{Definition}

\theoremstyle{remark}
\newtheorem{remark}[theorem]{Remark}
\numberwithin{equation}{section}

\begin{document}
\setcounter{page}{1}

\title[ Pseudo-differential operators   in H\"older spaces ]{ Pseudo-differential operators in H\"older spaces revisited. Weyl-H\"ormander calculus and Ruzhansky-Turunen classes.}

\author[D. Cardona]{Duv\'an Cardona}
\address{
  Duv\'an Cardona:
  \endgraf
  Department of Mathematics  
  \endgraf
  Pontificia Universidad Javeriana.
  \endgraf
  Bogot\'a
  \endgraf
  Colombia
  \endgraf
  {\it E-mail address} {\rm duvanc306@gmail.com;
cardonaduvan@javeriana.edu.co}
  }

\subjclass[2010]{Primary {}.}


\begin{abstract}
In this work we obtain  continuity results on H\"older spaces for operators belonging to a Weyl-H\"ormander calculus for metrics such that the class of the associated operators contains, in particular,  certain hypoelliptic laplacians. With our results we recover some historical H\"older boundedness theorems (see R. Beals \cite{Be,Be2}). The action of (periodic) Ruzhansky-Turunen classes of pseudo-differential operators on H\"older spaces also will be investigated.
\end{abstract} \maketitle
\tableofcontents

\section{Introduction}
\subsection{Outline of the paper}
For every $0<s<1,$ the (Lipschitz) H\"older space $\Lambda^s(\mathbb{R}^n)$ consists of those functions $f$ satisfying
\begin{equation}\label{Holder}
  \Vert f\Vert_{\Lambda^s}:=\sup_{x,y\in\mathbb{R}^n}|f(x+y)-f(y)||x|^{-s}<\infty.
\end{equation} In this work we study pseudo-differential operators on H\"older spaces. These are linear  operators of the  form
\begin{equation}\label{pseudo}
Af(x)\equiv  \sigma(x,D_x)f(x):=\int_{\mathbb{R}^n}e^{i2\pi x\cdot \xi}\sigma(x,\xi)\hat{f}(\xi)d\xi,\,\,\,f\in C^{\infty}_0(\mathbb{R}^n),
\end{equation}
where the function $\widehat{f}$ is the Fourier transform of $f$ and the function $\sigma$ is the so called,  symbol associated to  the operator $\sigma(x,D_x).$ In this paper we give mapping properties   for pseudo-differential operators  (with symbols associated to Weyl-H\"ormander classes) on H\"older spaces. These classes usually denoted by $S(m,g)$ are associated to a H\"ormander metric $g=\{g_{(x,\xi)}:(x,\xi)\in\mathbb{R}^{2n}\}$ on the phase space $\mathbb{R}^n\times \mathbb{R}^n$ and  a weight $m$ on $\mathbb{R}^{2n}.$ The particular case $$g=g^{\rho,\delta}:=\langle \xi\rangle^{2\delta}dx^{2}+\langle \xi\rangle^{-2\rho}d\xi^{2},\,\,m(x,\xi)=\langle \xi\rangle^{m},\,\,0\leq \delta\leq \rho\leq 1,\,\delta<1,$$
where $\langle \xi\rangle:=(1+|\xi|^2)^{\frac{1}{2}},$ corresponds to the $(\rho,\delta)-$H\"ormander class $S^{m}_{\rho,\delta}(\mathbb{R}^{2n}),$ which consists of those symbols satisfying
\begin{equation}\label{RDHC}
  |\partial_{x}^{\beta}\partial_\xi^{\alpha}\sigma(x,\xi)|\leq C_{\alpha,\beta}\langle \xi\rangle^{m-\rho|\alpha|+\delta|\beta|}.
\end{equation}
These classes were introduced by L. H\"ormander in 1967 motivated by the study of hypoelliptic problems as the heat equation. Another important case are the Shubin classes $\Sigma^m(\mathbb{R}^n\times \mathbb{R}^n)$ defined by the condition
\begin{equation}\label{Shubin}
  |\partial_{x}^{\beta}\partial_\xi^{\alpha}\sigma(x,\xi)|\leq C_{\alpha,\beta}\langle x,\xi\rangle^{m-\rho(|\alpha|+|\beta|)},
\end{equation}
where $\langle x, \xi\rangle:=(1+|x|^2+|\xi|^2)^{\frac{1}{2}},$ which from the Weyl-H\"ormander calculus can be obtained with
$$  g=g^{\rho}:=\langle x,\xi \rangle^{-\rho}(dx^2+d\xi^2),\,\,\,m(x,\xi)=\langle x,\xi \rangle^{m},\,\,\,0\leq \rho\leq 1.$$
For symbol classes associated to anharmonic oscillators we refer the reader to Chatzakou, Delgado and Ruzhansky \cite{ChatDelRuz}.
The quantisation process of pseudo-differential operators associates  to each function $\sigma(x,\xi)$ in  suitable classes (of symbols) on the phase space $\mathbb{R}^n\times \mathbb{R}^n,$ a (densely defined) linear operator $\sigma(x,D_{x})$ on the Hilbert space $L^2(\mathbb{R}^n),$ such that the coordinate functions $x_i$ and $\xi_i$ correspond to the operators $x_{j}$ and $D_{j}=-i2\pi\frac{\partial}{\partial x_j},$   and such that the properties of the symbols $\sigma(x,\xi)$ (positivity, boundedness, differentiability, invertibility, homogeneity, integrability, etc.) are reflected in some sense in the properties of the operators $\sigma(x,D_{x}),$ (positivity, mapping properties, invertibility, Fredholmness, geometric information, compactness). To our knowledge, the first general quantization procedure, (in pseudo-differential operators theory) and in many ways still the most satisfactory one, was proposed by Weyl (see H\"ormander\cite{Hor2}) not long after the invention of quantum mechanics.

The symbols $\sigma(x,\xi)$ considered in this work, belong to Weyl-H\"ormander classes $S(m^{-n\varepsilon},g),$ which are associated to the metric
$
  g_{(x,\xi)}(dx,d\xi)=m(x,\xi)^{-2}(\langle \xi \rangle^2 dx^2+d\xi^2),$ $x,\xi\in\mathbb{R}^n,
$
and  the weight
$
  m(x,\xi)=(a(x,\xi)+\langle \xi\rangle)^{\frac{1}{2}}.
$
The symbol $a,$ which is assumed positive and classical, is the principal part of the second order differential operator $L$ defined by
\begin{equation}\label{ST}
  {Lf}=-\sum_{ij}a_{ij}(x)\frac{\partial^2}{\partial{x_i\partial x_j}}f+[\textnormal{partial derivatives of lower order }]f,\,\,\, f\in C_{0}(\mathbb{R}^n).
\end{equation}
We also assume for every $x\in\mathbb{R}^n,$ that the matrix $A(x)=(a_{ij}(x))$ is a positive, semi-definite matrix  of rank $r(x):=\textnormal{rank}(A(x))\geq r_0\geq 1.$ The coefficients $a_{ij}$ are  smooth functions which are uniformly bounded on $\mathbb{R}^n$, together with all their derivatives. Important examples  arise from operators of the form
\begin{equation}\label{Ex1}
  L=-\sum_{j}X_{j}^{*}X_{j}+X_0,\,\,\,L=-\sum_{j}X_{j}^{*}X_{j},
\end{equation} where $\{X_i\}$ is a system of  vector fields  on $\mathbb{R}^n$ satisfying the H\"ormander conditions of order 2 (this means that the vector fields $X_i$ and their commutators provide a basis of $\mathbb{R}^n,$ or equivalently that $\mathbb{R}^{n}=\textnormal{Lie}\{X_i\}_{i}$).

On the other hand,   H\"older spaces on the torus $\mathbb{T}^n=\mathbb{R}^n/\mathbb{Z}^n,$ are the Banach spaces defined for each $0<s\leq 1$ by $$\Lambda^{s}(\mathbb{T}^n)=\{f:\mathbb{T}^n\rightarrow \mathbb{C} : |f|_{\Lambda^s}=\sup_{x,h\in \mathbb{T}^n} {|f(x+h)-f(x)|} {|h|^{-s}} <\infty  \}$$ 
together with the norm $\Vert f\Vert_{\Lambda^{s}}=|f|_{\Lambda^s}+\sup_{x\in \mathbb{T}^n}|f(x)|.$ In the second part of this work we investigate the action of periodic operators on  toroidal H\"older spaces. These  operators have  the form
 \begin{equation}
 a(x,D_{x})f(x)\equiv \textnormal{Op}(a)f(x)=\sum_{\eta\in \mathbb{Z}^n}e^{i2\pi \langle x, \eta\rangle}a(x,\eta)\hat{f}(\eta),
\end{equation}
where the function $a(x,\eta)$ is called, the symbol of $a(x,D_{x}),$ and $\widehat{f}$ is the periodic Fourier transform of $f.$ 
Periodic pseudo-differential operators were introduced by M. Agranovich  \cite{ag} who proposed a global quantization of periodic pseudo-differential operators on the one dimensional torus $\mathbb{R}/\mathbb{Z}\equiv\mathbb{T}$. Later, this theory was widely developed by G. Vainikko,  M. Ruzhansky and V. Turunen (see \cite{Ruz}), and subsequently  generalised on  compact Lie groups. For some recent work on boundedness results of periodic pseudo-differential operators  we refer the reader to the references \cite{Duvan2,Duvan3,periodic,Profe2, s2, Ruz, Ruz-2} where the subject was treated in $L^p$-spaces. Pseudo-differential operators with symbols in H\"ormander classes can be defined on smooth closed manifolds by using local charts (localizations). Agranovich (see \cite{ag}) gives a global definition of pseudo-differential operators on the circle $\mathbb{S}^1,$ (instead of the local formulation on the circle as a manifold). By using the Fourier transform Agranovich's definition was readily generalizable to the $n-$dimensional torus $\mathbb{T}^n.$ Indeed, G. Vainikko, M. Ruzhansky and V. Turunen introduced H\"ormander classes on any $n$-torus $\mathbb{T}^n$ of the following way:  given $m\in \mathbb{R},\, 0\leq \rho,\delta \leq 1,$ we say that $a \in S^{m}_{\rho,\delta}(\mathbb{T}^n\times \mathbb{Z}^n)$ ($(\rho,\delta)-$ symbol classes) if 
 \begin{equation}
\forall \alpha,\beta\in\mathbb{N}^n,\exists C_{\alpha,\beta}>0,\,\, |\Delta^{\alpha}_{\xi}\partial^{\beta}_{x}a(x,\xi)|\leq C_{\alpha,\beta}\langle \xi \rangle^{m-\rho|\alpha|+\delta|\beta|}.
 \end{equation} Here $\Delta_\xi$ is the usual difference operator on sequences defined by
 \begin{equation}
 \Delta_{\xi_j} a(\xi):=a(\xi+e_j)-a(\xi);\,\,\,\,e_{j}(k):=\delta_{jk},\,\,\,\Delta^{\alpha}_{\xi}=\Delta_{\xi_1}^{\alpha_1}\cdots \Delta_{\xi_n}^{\alpha_n},
 \end{equation}
where $\alpha=(\alpha_1,\cdots,\alpha_n)\in\mathbb{N}_0^n.$
 It is a non-trivial result, that the  pseudo-differential  calculus with symbols on $(\rho,\delta)-$classes by Vainikko-Ruzhansky-Turunen  and the periodic H\"ormander calculus via localisations, are equivalent, this means that
$$ \Psi^m_{\rho,\delta}(\mathbb{T}^n\times \mathbb{Z}^n)=\Psi^m_{\rho,\delta}(\mathbb{T}^n\times \mathbb{R}^n) . $$ 
This fact is known as the  McLean equivalence theorem (see \cite{Mc}).  The principal consequence of the McLean equivalence is that we can transfer the boundedness properties of pseudo-differential operators on $\mathbb{R}^n$ (in H\"ormander classes) to pseudo-differential operators on the torus. So, if we denote by $\Lambda^s(\mathbb{R}^n)$  the H\"older space with regularity order $s,$ $0<s<1,$ then it is well known that $T_\sigma\in \Psi_{1,0}^0(\mathbb{R}^n\times \mathbb{R}^n )$   extends to a bounded operator on $\Lambda^s(\mathbb{R}^n)$ (see E. Stein \cite{Eli}, p. 253).
 In view of the McLean equivalence, if $\sigma(x,D_x)\in \Psi^0_{1,0}(\mathbb{T}^n\times \mathbb{Z}^n) $ is a periodic operator, then \begin{eqnarray}
\sigma(x,D_x):\Lambda^s(\mathbb{T}^n)\rightarrow \Lambda^s(\mathbb{T}^n),\,\,0<s<1,
\end{eqnarray} 

 extends to a bounded operator.  So,  the problem of the boundedness of periodic pseudo-differential operators can be reduced to consider symbols of limited regularity.

\subsection{State of the art and  main results}
In general, classes of bounded pseudo-differential operators on Lebesgue spaces $L^p{(\mathbb{R}^n)},$ $1<p<\infty,$ also provide bounded operators on H\"older spaces $\Lambda^s(\mathbb{R}^n),$ Besov spaces $B^s_{p,q}(\mathbb{R}^n),$ Triebel-Lizorkin spaces $F^s_{p,q}(\mathbb{R}^n),$ and several types of function spaces. Through of the metric $g^{1,0}_{(x,\xi)}(dx,d\xi)=dx^2+\langle\xi\rangle^{-2}d\xi^2,$ and the weight $m(x,\xi)=1$ we recover the Kohn-Nirenberg class of order zero $S^0(\mathbb{R}^n)=S(1,g^{1,0})$ which give bounded pseudo-differential operators  on $L^p$ spaces,  $1<p<\infty,$ and on  every H\"older space $\Lambda^s.$ In fact, for certain metrics, R. Beals showed that the classes $S(1,g)$ give bounded pseudo-differential operators  in $L^p$ spaces as well as in H\"older spaces (see the classical references  Beals \cite{Be} and Beals \cite{Be2}; the case of periodic operators was treated in \cite{Car,Car1}).   On the other hand, if we define
\begin{equation}\label{q0}
  Q_0:=r_0+2(n-r_0),\,\,\varepsilon_0:=\frac{Q_0}{2n}-\frac{1}{2},
\end{equation} it was proved by J. Delgado in \cite{Profe2}, that the set of operators with symbols in the classes $S(m^{-n\varepsilon},g)$ provides bounded operators on $L^p$-spaces for all $1<p<\infty$ and $\varepsilon_0\leq \varepsilon<\frac{Q_0}{2n}.$ Moreover,  for $0\leq \beta<\varepsilon_0,$ the operators in the classes $S(m^{-\beta},g)$ are $L^p$-bounded provided that $|1/2-1/p|\leq \beta/2n\varepsilon_0.$ These are important extensions  of some results by C. Fefferman and R. Beals. In this paper we will prove the following.
\begin{itemize}
\item Let us assume $\varepsilon_0\leq \varepsilon<\frac{Q_0}{2n}$ and let $\sigma\in S(m^{-n\varepsilon},g).$  Then 
the pseudo-differential operator $\sigma(x,D_x)$ extends to a bounded operator on $\Lambda^s(\mathbb{R}^n)$ for all $0<s<1.$ 
\item Let us assume $0\leq \beta \leq n\varepsilon_0$ and let $\sigma\in S(m^{-\beta},g).$  Then 
the pseudo-differential operator $\sigma(x,D_x)$ extends to a bounded operator from  $\Lambda^s(\mathbb{R}^n)$ into  $\Lambda^{s-(n\varepsilon_0-\beta)}(\mathbb{R}^n)$  for all $0<s<1$ provided that $0<s-(n\varepsilon_0-\beta)<1.$ 
\end{itemize}  In particular, if the operator $L$ is elliptic we have $Q_0=2n,$ $\varepsilon_0=0,$ and then we recover the classical H\"older estimate by R. Beals for $S(1,g)$-classes (see Beals \cite{Be,Be2}).

On the other hand, our results for periodic pseudo-differential operators can be summarised as follow.
\begin{itemize} \item  Let $0\leq \varepsilon <1$ and $k:=[\frac{n}{2}]+1,$  let $\sigma:\mathbb{T}^n\times \mathbb{Z}^n\rightarrow \mathbb{C}$ be a symbol such that  $|\Delta_{\xi}^{\alpha}\sigma(x,\xi)|\leq C_{\alpha}\langle \xi \rangle^{-\frac{n}{2}\varepsilon-(1-\varepsilon)|\alpha|},$  for $|\alpha| \leq k.$ Then  $\sigma(x,D):\Lambda^s(\mathbb{T}^n)\rightarrow\Lambda^s(\mathbb{T}^n)$ extends to a bounded linear operator for all $0<s<1.$ 
\item Let $0<\rho\leq 1,$ $0\leq \delta\leq 1,$ $\ell\in\mathbb{N},$ $k:=[\frac{n}{2}]+1,$  and let $A:C^{\infty}(\mathbb{T}^n)\rightarrow C^{\infty}(\mathbb{T}^n)$ be a pseudo-differential operator with symbol $\sigma$ satisfying  
\begin{equation}\label{eqMaint'}
\vert \partial_x^\beta\Delta_{\xi}^{\alpha}\sigma(x,\xi)\vert\leq C_{\alpha}\langle \xi \rangle^{-m-\rho|\alpha|+\delta|\beta|}
\end{equation}
 for all $|\alpha| \leq k,$ $|\beta|\leq \ell.$ Then  $A:\Lambda^s(\mathbb{T}^n)\rightarrow \Lambda^s(\mathbb{T}^n)$ extends to a bounded linear operator for all $0<s<1$ provided that $m\geq \delta \ell+\frac{n}{2}(1-\rho).$
\end{itemize}
The periodic conditions above  provide Fefferman type conditions for the H\"older boundedness of periodic operators. These estimates extend the H\"older results in \cite{Car} and \cite{Car1}.

This paper is organized as follows. In Section \ref{preliminaries} we present some basics on Weyl-H\"ormander classes and the periodic Ruzhansky-Turunen classes.  Finally, in Section \ref{proofs} we prove our main results. 
\section{Weyl-H\"ormander and Ruzhansky-Turunen classes}\label{preliminaries}
\subsection{Weyl-H\"ormander classes: $S(m,g)$}
Now, we provide some notions on the Weyl-H\"ormander calculus. An extensive treatment on the subject can be found in H\"ormander \cite{Hor2}, but we present only the necessary elements in order to get a simple  presentation.  The Kohn-Nirenberg quantisation procedure (or classical quantisation) associates to every $\sigma\in\mathscr{S}'(\mathbb{R}^{2n})$ the operator
\begin{equation}\label{cl}
  \sigma(x,D_x)f(x):=\int_{\mathbb{R}^n}e^{i2\pi x\cdot \xi}\sigma(x,\xi)\widehat{f}(\xi)d\xi=\int_{\mathbb{R}^n}\int_{\mathbb{R}^n}e^{i2\pi (x-y)\cdot \xi}\sigma(x,\xi)d\xi f(y)dy,
\end{equation}
where  $f\in C^{\infty}_0(\mathbb{R}^n). $ On the other hand the Weyl quantization of $\sigma\in\mathscr{S}'(\mathbb{R}^{2n})$ is given by the operator
\begin{equation}\label{Wq}
 \sigma^{\omega}(x,D_x)f(x):= \int_{\mathbb{R}^n}\int_{\mathbb{R}^n}e^{i2\pi (x-y)\cdot \xi}\sigma(\frac{1}{2}(x+y),\xi)d\xi f(y)dy,\,\,\,f\in C^{\infty}_0(\mathbb{R}^n).
\end{equation}
There exists a relation between the classic quantization and the Weyl quantization (see \cite{Profe3,Profe2}). For the Weyl quantisation, the composition rule can be formulated in terms of the symplectic form $\langle (x,\xi),(y,\eta)\rangle_\omega:=y\cdot \xi-x\cdot \eta,$ trough of the  operation
\begin{equation}\label{operation}
  a\# b(X):=\frac{1}{\pi^{2n}}\int_{\mathbb{R}^{2n}}\int_{\mathbb{R}^{2n}}e^{-2i \langle X-Y,X-Z\rangle_{\omega}}a(Y)b(Z)dYdZ,\,\,\,\,\,X=(x,\xi).
\end{equation}
satisfying the identity
\begin{equation}\label{identitycompo}
  (a\#b)^\omega=a^\omega\circ b^\omega.
\end{equation}
$S(m,g)$-classes are associated to H\"ormander metrics which are  special Riemannian metrics $g=(g_X)_{X\in \mathbb{R}^{2n}}$ on $\mathbb{R}^{2n},$ satisfying the following properties.
\begin{itemize}
  \item{Continuity}: there exist positive constants $C,c,c'$ such that $g_{X}(Y)\leq C$ implies \begin{equation}\label{continuity}
                c'\cdot g_{X+Y}(Z)\leq g_{X}(Z)\leq c\cdot g_{X+Y}(Z), \textnormal{   where   } X,Y,Z\in\mathbb{R}^{2n}.
              \end{equation}
  \item{Uncertainty principle}: for $X,Y,T,Z\in\mathbb{R}^{2n}$ and
  \begin{equation}\label{Incer}
    g_{X}^{\langle\cdot,\cdot \rangle_{\omega}}(T)=\sup_{Z\neq 0}\frac{  \langle T,Z \rangle_{\omega}^2}{g_X(Z)}
  \end{equation} the metric $g$ satisfies the uncertainty principle if
  \begin{equation}\label{eqincert}
  \lambda_g(X):=  \inf_{T\neq 0}(\frac{g_{X}^{\langle\cdot,\cdot \rangle_{\omega}}(T)  }{g_X(T)})^{\frac{1}{2}}\geq 1.
  \end{equation}
  \item{Temperancy}: we say that $g$ is a temperate metric if there exists $\overline{C}>0$ and $J\in\mathbb{N}$ satisfying
      \begin{equation}\label{temperancy}
        (\frac{g_X(\cdot)}{g_Y(\cdot)})^{\pm 1}\leq \overline{C}(1+g_Y^{\langle\cdot,\cdot\rangle_\omega}(X-Y))^{J}.
      \end{equation}
\end{itemize}
\ On the other hand, classes $S(m,g)$ associated to H\"ormander metrics $g$ require the notion of $g$-admissible weight $M,$ which are  strictly positive functions satisfying
\begin{itemize}
  \item{Continuity:} there exists $D>0$ such that
  \begin{equation}\label{continuityweigh}
    (\frac{M(X+Y)}{M(X)})^{\pm 1}\leq D,\,\,\,\,\textnormal{if} \,\,\,\,g_{X}(Y)\leq \frac{1}{D}.
  \end{equation}

  \item{Temperancy:} there exists $D'>0$ and $N_0\in\mathbb{N}$ such that
  \begin{equation}\label{temperancyweigh}
    (\frac{M(Y)}{M(X)})^{\pm 1}\leq D'(1+g^{ \langle\cdot,\cdot\rangle_\omega }_{Y}(X-Y))^N,\,\,\,\, \textnormal{if}  \,\,\,\,g_{X}(Y)\leq \frac{1}{D'}.
  \end{equation}
  
\end{itemize}
We end this section with the definition of $S(m,g)$ Weyl-H\"ormander classes.

\begin{definition}[Weyl-H\"ormander classes]
  Let us assume that $g$ is a  H\"ormander metric and let $m$ be a $g$-admissible weight. The class $S(m,g)$ consists of those smooth functions $\sigma$ on $\mathbb{R}^{2n}$ satisfying the symbol inequalities
  \begin{equation}\label{Smgdefi}
    |\sigma^{(k)}(T_1,T_2,\cdots,T_k)|\leq C_k M(X) \cdot (g_X(T_1)g_X(T_2)\cdots g_X(T_k)))^{\frac{1}{2}}
  \end{equation}
\end{definition}
For every symbol $\sigma\in S(m,g)$ we denote by $\Vert \sigma \Vert_{S(m,g),k}$ the minimum $C_k$ satisfying the inequality \eqref{Smgdefi}.
\subsection{Ruzhansky-Turunen classes}

The toroidal H\"older spaces are the Banach spaces defined for each $0<s\leq 1$ by $$\Lambda^{s}(\mathbb{T}^n)=\{f:\mathbb{T}^n\rightarrow \mathbb{C} : |f|_{\Lambda^s}=\sup_{x,h\in \mathbb{T}^n} {|f(x+h)-f(x)|} {|h|^{-s}} <\infty  \}$$ 
together with the norm $\Vert f\Vert_{\Lambda^{s}}=|f|_{\Lambda^s}+\sup_{x\in \mathbb{T}^n}|f(x)|.$

In our analysis of periodic operators on H\"older spaces we use the standard notation  (see \cite{ Hor1, Hor2, Ruz, Wong}). The discrete Schwartz space $\mathcal{S}(\mathbb{Z}^n)$ denote the space of discrete functions $\phi:\mathbb{Z}^n\rightarrow \mathbb{C}$ such that 
 \begin{equation}
 \forall M\in\mathbb{R}, \exists C_{M}>0,\, |\phi(\xi)|\leq C_{M}\langle \xi \rangle^M,
 \end{equation}
where $\langle \xi \rangle=(1+|\xi|^2)^{\frac{1}{2}}.$ The toroidal Fourier transform is defined for any $f\in C^{\infty}(\mathbb{T}^n)$ by $$\hat{f}(\xi)=\int_{\mathbb{T}^n}e^{-i2\pi\langle x,\xi\rangle}f(x)dx,\,\,\xi\in\mathbb{Z}^n,\,\,\,\,\langle x,\xi\rangle=x_1\xi_1+\cdots +x_n\xi_n.$$ The Fourier inversion formula is given by $$f(x)=\sum_{\xi\in\mathbb{Z}^n}e^{i2\pi\langle x,\xi \rangle }\hat{u}(\xi),\,\,x\in\mathbb{T}^n.$$ Now, the periodic H\"ormander class $S^m_{\rho,\delta}(\mathbb{T}^n\times \mathbb{R}^n), where\,\, 0\leq \rho,\delta\leq 1,$ consists of those complex (1-periodic) functions $a(x,\xi)$ in $x$, which are smooth in $(x,\xi)\in \mathbb{T}^n\times \mathbb{R}^n$ and which satisfy toroidal symbols inequalities
\begin{equation}\label{css}
|\partial^{\beta}_{x}\partial^{\alpha}_{\xi}a(x,\xi)|\leq C_{\alpha,\beta}\langle \xi \rangle^{m-\rho|\alpha|+\delta|\beta|}.
\end{equation}
Symbols in $S^m_{\rho,\delta}(\mathbb{T}^n\times \mathbb{R}^n)$ are symbols in $S^m_{\rho,\delta}(\mathbb{R}^n\times \mathbb{R}^n)$ (see \cite{Hor1, Ruz}) with order $m$ which are 1-periodic in $x.$
If $a(x,\xi)\in S^{m}_{\rho,\delta}(\mathbb{T}^n\times \mathbb{R}^n),$ the corresponding pseudo-differential operator is defined by
\begin{equation}\label{hh}
a(x,D_x)u(x)=\int_{\mathbb{T}^n}\int_{\mathbb{R}^n}e^{i2\pi\langle x-y,\xi \rangle}a(x,\xi)u(y)d\xi dy,\,\, u\in C^{\infty}(\mathbb{T}^n).
\end{equation}
The set $S^m_{\rho,\delta}(\mathbb{T}^n\times \mathbb{Z}^n),\, 0\leq \rho,\delta\leq 1,$ consists of  those functions $a(x, \xi)$ which are smooth in $x$  for all $\xi\in\mathbb{Z}^n$ and which satisfy

\begin{equation}\label{cs}
\forall \alpha,\beta\in\mathbb{N}^n,\exists C_{\alpha,\beta}>0,\,\, |\Delta^{\alpha}_{\xi}\partial^{\beta}_{x}a(x,\xi)|\leq C_{\alpha,\beta}\langle \xi \rangle^{m-\rho|\alpha|+\delta|\beta|}.
 \end{equation}

The operator $\Delta_\xi^\alpha$ in  \eqref{cs} is the difference operator which is defined as follows. First, if $f:\mathbb{Z}^n\rightarrow \mathbb{C}$ is a discrete function and $(e_j)_{1\leq j\leq n}$ is the canonical basis of $\mathbb{R}^n,$
$
(\Delta_{\xi_{j}} f)(\xi)=f(\xi+e_{j})-f(\xi).
$
If $k\in\mathbb{N},$ denote by $\Delta^k_{\xi_{j}}$  the composition of $\Delta_{\xi_{j}}$ with itself $k-$times. Finally, if $\alpha\in\mathbb{N}^n,$ $\Delta^{\alpha}_{\xi}= \Delta^{\alpha_1}_{\xi_{1}}\cdots \Delta^{\alpha_n}_{\xi_{n}}.$

The main object here are the toroidal operators (or periodic operators) with symbols $a(x,\xi).$ They are defined as (in the sense of Vainikko, Ruzhansky and Turunen)
\begin{equation}\label{aa}
a(x,D_x)u(x):=\textnormal{Op}(a)u=\sum_{\xi\in\mathbb{Z}^n}e^{i 2\pi\langle x,\xi\rangle}a(x,\xi)\hat{u}(\xi),\,\, u\in C^{\infty}(\mathbb{T}^n).
\end{equation}

The relation on toroidal and euclidean symbols will be explain as follows. There exists a process to interpolate the second argument of symbols on $\mathbb{T}^n\times \mathbb{Z}^n$ in a smooth way to get a symbol defined on $\mathbb{T}^n\times \mathbb{R}^n.$
\begin{proposition}\label{eq}
Let $0\leq \delta \leq 1,$ $0< \rho\leq 1.$ The symbol $a\in S^m_{\rho,\delta}(\mathbb{T}^n\times \mathbb{Z}^n)$ if and only if there exists  a Euclidean symbol $a'\in S^m_{\rho,\delta}(\mathbb{T}^n\times \mathbb{R}^n)$ such that $a=a'|_{\mathbb{T}^n\times \mathbb{Z}^n}.$
\end{proposition}
\begin{proof} The proof can be found in \cite{Mc, Ruz}.
\end{proof}

It is a non trivial fact, however, that the definition of pseudo-differential operator on a torus given by Agranovich (equation \ref{aa} )  and H\"ormander (equation \ref{hh}) are equivalent. McLean (see \cite{Mc}) prove this result for all the H\"ormander classes $S^m_{\rho,\delta}(\mathbb{T}^n\times \mathbb{Z}^n).$ A different proof to this fact can be found in \cite{Ruz}, Corollary 4.6.13 by using a periodisation technique.

\begin{proposition}$\label{eqc}($Equality of classes$).$ For every $0\leq \delta \leq 1$ and $0<\rho\leq 1,$ we have $\Psi^{m}_{\rho,\delta}(\mathbb{T}^n\times \mathbb{Z}^n)=\Psi^{m}_{\rho,\delta}(\mathbb{T}^n\times \mathbb{R}^n).$
\end{proposition}

 For the multilinear aspects of the theory of periodic operators we refer the reader to Cardona and Kumar \cite{CardonaKumar}.

\section{Pseudo-differential operators in H\"older spaces }\label{proofs}
\subsection{$S(m,g)$-classes in H\"older spaces}
Now we prove our H\"older estimates. Our main tool will be the characterisation of H\"older spaces in terms of dyadic decompositions.

The symbols $\sigma(x,\xi)$ considered in this section, belong to Weyl-H\"ormander classes $S(m^{-n\varepsilon},g),$ which are associated to the metric
$$g_{(x,\xi)}(dx,d\xi)=m(x,\xi)^{-2}(\langle \xi \rangle^2 dx^2+d\xi^2),$$ $x,\xi\in\mathbb{R}^n,
$
and  the weight
$
  m(x,\xi)=(a(x,\xi)+\langle \xi\rangle)^{\frac{1}{2}}.
$
The symbol $a,$ which is assumed positive and classical, is the principal part of the second order differential operator $L$ defined by
\begin{equation}\label{ST'}
  {Lf}=-\sum_{ij}a_{ij}(x)\frac{\partial^2}{\partial{x_i\partial x_j}}f+[\textnormal{partial derivatives of lower order }]f,\,\,\, f\in C_{0}(\mathbb{R}^n).
\end{equation}
We also assume for every $x\in\mathbb{R}^n,$ that the matrix $A(x)=(a_{ij}(x))$ is a positive, semi-definite matrix  of rank $r(x):=\textnormal{rank}(A(x))\geq r_0\geq 1.$ The coefficients $a_{ij}$ are  smooth functions which are uniformly bounded on $\mathbb{R}^n$, together with all their derivatives. Important examples  arise from operators of the form
\begin{equation}\label{Ex1'}
  L=-\sum_{j}X_{j}^{*}X_{j}+X_0,\,\,\,L=-\sum_{j}X_{j}^{*}X_{j},
\end{equation} where $\{X_i\}$ is a system of  vector fields  on $\mathbb{R}^n$ satisfying the H\"ormander condition of order 2 (this means that the vector fields $X_i$ and their commutators provide a basis of $\mathbb{R}^n,$ or equivalently that $\mathbb{R}^{n}=\textnormal{Lie}\{X_i\}_{i}$).
\begin{remark}Classical examples of operators as in \eqref{Ex1} are the following:
\begin{itemize}
 \item the Laplacian $-\Delta_{x}:=-\sum_{i=1}^{n}\partial_{x_i}^2$ on $\mathbb{R}^{n}.$ \item    The heat operator  $-\Delta_x+\partial_{t}$ on $\mathbb{R}^{n+1}.$
     \item The Mumford operator on $\mathbb{R}^4:$
\begin{equation}\label{mumford}
  M=-X_{1}^2-X_{0}=-\partial_\theta^{2}+\cos(\theta)\partial_x-\sin(\theta)\partial_y+\partial_t,\,\,X_1=\partial_\theta,
\end{equation}
where we have denoted by $(\theta,x,y,t)$ the coordinates of $\mathbb{R}^4.$ In this case, $X_0=-M-X_1,$ $X_2=[X_1,X_0],$ $X_3=[X_1,X_2]$ and $\textnormal{span}\{ X_0,X_1,X_2,X_3\}=\mathbb{R}^4,$ which shows that $M$ satisfies the H\"ormander condition of order 2.
\item The Kolmogorov operator on $\mathbb{R}^{3}:$
\begin{equation}\label{Kolmogorov}
  K=-X_1^{2}-X_0=-\partial_{x}^{2}-x\partial_{y}+\partial_t,\,\,X_{1}=\partial_x.
\end{equation} A similar analysis as in the Mumford operator shows that $K$ satisfies the H\"ormander condition of order 2.
\item The operator 
\begin{equation} L=\frac{\partial^2}{\partial t^2} +\frac{\partial^2}{\partial x^2}+e^{-\frac{1}{|x|^\delta}}\frac{\partial^2}{\partial y^2} \end{equation} on $\mathbb{R}^3.$
\end{itemize}
\end{remark}
Our starting point is the following lemma which is a slight variation of one proved by J. Delgado (see Lemma 3.2 of \cite{Profe2}) and whose proof is only an obvious variation of  Delgado's proof. We will use the following parameters
\begin{equation}\label{q0}
  Q_0:=r_0+2(n-r_0),\,\,\varepsilon_0:=\frac{Q_0}{2n}-\frac{1}{2}.
\end{equation} 
\begin{lemma} \label{lemma}
Let us assume $\varepsilon_0\leq \varepsilon<\frac{Q_0}{2n}$ and let $\sigma\in S(m^{-n\varepsilon},g)$ supported in $R\leq a(x,\xi)+\langle \xi\rangle\leq \omega' R$ for $\omega',R>1.$ Then 
\begin{equation}
\Vert \sigma(x,D_x)f\Vert_{L^\infty(\mathbb{R}^n)}\leq C \Vert \sigma\Vert_{[\frac{n}{2}]+1, S(m^{-n\varepsilon},g) }\Vert f \Vert_{L^\infty(\mathbb{R}^n)},
\end{equation} holds true for every $f\in L^{\infty}(\mathbb{R}^n).$ Moreover, the constant $C>0$ does not depend on the parameters $R$ and $\omega'.$
\end{lemma}

Now we will prove our main theorem.

\begin{theorem} \label{MainT}
Let us assume $\varepsilon_0\leq \varepsilon<\frac{Q_0}{2n}$ and let $\sigma\in S(m^{-n\varepsilon},g).$  Then 
the pseudo-differential operator $\sigma(x,D_x)$ extends to a bounded operator on $\Lambda^s(\mathbb{R}^n)$ for all $0<s<1.$ Moreover, there exists $\ell$ such that \begin{equation}
\Vert \sigma(x,D_x)f\Vert_{\Lambda^s(\mathbb{R}^n)}\leq C \Vert \sigma\Vert_{\ell, S(m^{-n\varepsilon},g) }\Vert f \Vert_{\Lambda^s(\mathbb{R}^n)}.
\end{equation}
\end{theorem}
\begin{remark}
If  we  suppose that $L$ is an elliptic operator, then we deduce that $Q_0=2n,$ $\varepsilon_0=0.$ Consequently,  we recover the classical H\"older estimate due to R. Beals for $S(1,g)$-classes (see Theorem 4.1 of Beals \cite{Be} and Beals \cite{Be2}) and also the H\"older result for the class $S^0(\mathbb{R}^n)$ mentioned in the introduction.
\end{remark}

\begin{proof}[Proof of  Theorem  \ref{MainT}] Let us fix $f\in \Lambda^s(\mathbb{R}^n),$ $0<s<1.$ We will show that some positive  integer $\ell$ satisfies
\begin{equation}
\Vert \sigma(x,D_x)f\Vert_{\Lambda^s(\mathbb{R}^n)}\leq C \Vert \sigma\Vert_{\ell, S(m^{-n\varepsilon},g) }\Vert f \Vert_{\Lambda^s(\mathbb{R}^n)},
\end{equation}
for some positive constant $C$ independent of $f$. We split the proof in two parts.  In the first one, we prove the statement of the theorem for Fourier multipliers, i.e., pseudo-differential operators depending only on the Fourier variable $\xi.$ Later, in the second step, we extend the result for general pseudo-differential  operators.\\

\noindent{{\textit{Step 1.}}} Let us consider  the Bessel potential of first order $\mathcal{R}:=(I-\frac{1}{4\pi^2}\Delta_x)^{\frac{1}{2}}$, where $\Delta_x$ is the Laplacian on  $\mathbb{R}^n,$ and let us fix a dyadic
decomposition of its spectrum: we choose a function $\psi_0\in C^{\infty}_{0}(\mathbb{R}),$  $\psi_0(\lambda)=1,$  if $|\lambda|\leq 1,$ and $\psi(\lambda)=0,$ for $|\lambda|\geq 2.$ For every $j\geq 1,$ let us define $\psi_{j}(\lambda)=\psi_{0}(2^{-j}\lambda)-\psi_{0}(2^{-j+1}\lambda).$ Then we have
\begin{eqnarray}\label{deco1}
\sum_{l\in\mathbb{N}_{0}}\psi_{l}(\lambda)=1,\,\,\, \text{for every}\,\,\, \lambda>0.
\end{eqnarray}
We will use the following characterization for H\"older spaces in terms of dyadic decompositions (see Stein \cite{Eli}): $f\in \Lambda^s(\mathbb{R}^n),$ $0<s<1,$ if  and only if, 
\begin{equation}
\Vert f \Vert_{\Lambda^s(\mathbb{R}^n)}\asymp \Vert f\Vert_{B^s_{\infty,\infty}(\mathbb{R}^n)}:=\sup_{l\geq 0}2^{ls}\Vert \psi_l(\mathcal{R})f \Vert_{L^\infty(\mathbb{R}^n)}
\end{equation}
where $\psi_l(\mathcal{R})$ is defined by the functional calculus associated to the self-adjoint operator $\mathcal{R}.$ If $\sigma(D_x)=\sigma(x,D_x  )$ has a symbol depending only on the Fourier variable, then 
\begin{equation}
\Vert \sigma(D_x) f \Vert_{\Lambda^s(\mathbb{R}^n)}\asymp \Vert    \sigma(D_x)f\Vert_{B^s_{\infty,\infty}(\mathbb{R}^n)}:=\sup_{l\geq 0}2^{ls}\Vert \psi_l(\mathcal{R})\sigma(D_x)f \Vert_{L^\infty(\mathbb{R}^n)}.
\end{equation}
Let us note that $\sigma(D_x)$ commutes with $\psi_l(\mathcal{R})$ for every $l,$ that 
\begin{equation}
\psi_l(\mathcal{R})\sigma(D_x)= \sigma(D_x)\psi_l(\mathcal{R})=: \sigma_{l}(D_x),
\end{equation}
where  $\sigma_{l}(D_x)$ has a smooth symbol supported in $\{\xi:2^{l-1}\leq \langle \xi\rangle\leq 2^{l+1}   \}.$ Now let us observe that from the estimate $0\leq a(x,\xi)\lesssim \langle \xi \rangle^2,$ every  $\xi$ in the support of $\sigma_l(\xi)$ satisfies $4^{l-1}\leq  \langle \xi \rangle^2\leq 4^{l+1}$ and 
\begin{equation}
4^{l-1}\leq a(x,\xi)+\langle \xi \rangle\lesssim \langle \xi \rangle^2+\langle \xi \rangle\lesssim 4^{l+1}.
\end{equation}
This analysis shows that the support of $\sigma_{l}$ lies in the set 
\begin{equation}
\{ (x,\xi) \in\mathbb{R}^n_x\times \mathbb{R}^n_\xi:4^{l-1}\leq a(x,\xi)+\langle \xi \rangle \leq \omega 4^{l+1}\},
\end{equation}
for some $\omega>0$ which we can assume larger than one. So, 
by Lemma \ref{lemma} with $\omega'=4^2\omega$ and $R=4^{l-1},$ we deduce that $\sigma_{l}(D_x)$ is a bounded operator on $L^{\infty}(\mathbb{R}^n)$ with operator norm independent on $l.$ In fact,  $\sigma_{l}\in S(m^{-n\varepsilon},g)$ for all $l,$ and consequently
\begin{equation}
\Vert \sigma_{l}(D_x)\Vert_{\mathscr{B}(L^{\infty})}\leq C \Vert \sigma\Vert_{[\frac{n}{2}]+1, S(m^{-n\varepsilon},g) }.\end{equation} So, we have
\begin{align*}
\Vert \psi_l(\mathcal{R})\sigma(D_x) & f \Vert_{L^\infty(\mathbb{R}^n)} \\
&=\Vert   \sigma_l(D_x) \sum_{l'\in \mathbb{N}_{0}}\psi_{l'}(\mathcal{R})f \Vert_{L^\infty(\mathbb{R}^n)}  \\
&=\Vert   \sigma_l(D_x) \sum_{l'=l-1}^{l+1}\psi_{l'}(\mathcal{R})f \Vert_{L^\infty(\mathbb{R}^n)}  \\
&\leq \sum_{l'=l-1}^{l+1} \Vert \sigma_{l}(D_x)\Vert_{\mathscr{B}(L^{\infty})} \Vert    \psi_{l'}(\mathcal{R})f \Vert_{L^\infty(\mathbb{R}^n)} \\
& \lesssim \Vert \sigma\Vert_{[\frac{n}{2}]+1, S(m^{-n\varepsilon},g) } \sum_{l'=l-1}^{l+1}   \Vert    \psi_{l'}(\mathcal{R})f \Vert_{L^\infty(\mathbb{R}^n)}.
\end{align*}
As a consequence, we obtain
\begin{align*}\Vert \sigma(D_x) f \Vert_{\Lambda^s(\mathbb{R}^n)} &\asymp \sup_{l\geq 0}2^{ls}\Vert \psi_l(\mathcal{R})\sigma(D_x)f \Vert_{L^\infty(\mathbb{R}^n)}\\
& \lesssim \Vert \sigma\Vert_{[\frac{n}{2}]+1, S(m^{-n\varepsilon},g) }\sup_{l\geq 0} 2^{ls} \Vert    \psi_l(\mathcal{R})f \Vert_{L^\infty(\mathbb{R}^n)}. 
\end{align*}
Indeed, 
\begin{align*}
    &\sup_{l\geq 0}2^{ls}\Vert \psi_l(\mathcal{R})\sigma(D_x)f \Vert_{L^\infty(\mathbb{R}^n)}\\ &\lesssim \Vert \sigma\Vert_{[\frac{n}{2}]+1, S(m^{-n\varepsilon},g) }\sup_{l\geq 0}2^{ls} \sum_{l'=l-1}^{l+1}   \Vert    \psi_{l'}(\mathcal{R})f \Vert_{L^\infty(\mathbb{R}^n)}\\
    &= \Vert \sigma\Vert_{[\frac{n}{2}]+1, S(m^{-n\varepsilon},g) }\sup_{l\geq 0} \sum_{l'=l-1}^{l+1}  2^{(l-l')s} 2^{l's} \Vert    \psi_{l'}(\mathcal{R})f \Vert_{L^\infty(\mathbb{R}^n)}\\
    &\leq \Vert \sigma\Vert_{[\frac{n}{2}]+1, S(m^{-n\varepsilon},g) }\left( 2^{s} +1+2^{-s} \right)\sup_{l'\geq 0}2^{l's}\Vert    \psi_{l'}(\mathcal{R})f \Vert_{L^\infty(\mathbb{R}^n)}.
\end{align*} Finally, we finish the first step, by observing that
\begin{align*}
    &\sup_{l\geq 0}2^{ls}\Vert \psi_l(\mathcal{R})\sigma(D_x)f \Vert_{L^\infty(\mathbb{R}^n)}\\
    & \lesssim C\Vert \sigma\Vert_{[\frac{n}{2}]+1, S(m^{-n\varepsilon},g) }\sup_{l'\geq 0}\Vert    \psi_{l'}(\mathcal{R})f \Vert_{L^\infty(\mathbb{R}^n)} \asymp \Vert \sigma\Vert_{[\frac{n}{2}]+1, S(m^{-n\varepsilon},g) }\Vert f\Vert_{\Lambda^s(\mathbb{R}^n)},
\end{align*}where $C=\left( 2^{s} +1+2^{-s} \right).$\\

\noindent{{\textit{Step 2.}}} Now, we extend the H\"older boundedness from multipliers to pseudo-differential operators by adapting a technique developed by Ruzhansky and Turunen for operators on compact Lie groups, used in our setting, to the non-compact case of $\mathbb{R}^n$ (see Ruzhansky and Turunen \cite{Ruz} and Ruzhansky and Wirth \cite{Ruz3}). So,  let us define for every $z\in\mathbb{R}^n,$ the multiplier
\begin{center}
$ \sigma_z(D_{x})f(x)=\int_{\mathbb{R}^n}e^{i2\pi \langle x, \eta\rangle}\sigma(z,\eta)\hat{f}(\eta)d\eta.$
\end{center}
For every $x\in\mathbb{R}^n$ we have the equality,
\begin{center}
$ \sigma_{x}(D_{x})f(x)=\sigma(x,D_x)f(x),$
\end{center}
and we can estimate the H\"older norm of the function $\sigma(x,D_x)f,$ as follows
\begin{align*}
\Vert \sigma_{x}(D_x)f(x)\Vert_{\Lambda^s} &\asymp \sup_{l\geq 0}2^{ls} \textnormal{esssup}_{x\in\mathbb{R}^n}|\psi_l(\mathcal{R})\sigma_x(D_x)f(x) |\\
&\leq  \sup_{l\geq 0}2^{ls} \textnormal{esssup}_{x\in\mathbb{R}^n}   \sup_{z\in\mathbb{R}^n}|\psi_l(\mathcal{R})\sigma_z(D_x)f(x) | \\
&=  \sup_{l\geq 0}2^{ls} \textnormal{esssup}_{x\in\mathbb{R}^n}   \sup_{z\in\mathbb{R}^n}| \sigma_z(D_x)  \psi_l(\mathcal{R})f(x) | \\
&\leq \sup_{l\geq 0}2^{ls}  \textnormal{esssup}_{x\in\mathbb{R}^n}   \sup_{z\in\mathbb{R}^n} \textnormal{esssup}_{\varkappa\in\mathbb{R}^n} |\sigma_z(D_\varkappa)  \psi_l(\mathcal{R})f(\varkappa) | \\
&= \sup_{l\geq 0}2^{ls}   \sup_{z\in\mathbb{R}^n} \Vert \sigma_z(D_\varkappa)  \psi_l(\mathcal{R})f(\varkappa) \Vert_{L^\infty(\mathbb{R}^n_\varkappa)} .\\
\end{align*} From the estimate for the operator norm of multipliers proved in the first step, we deduce
\begin{align*}\sup_{z\in\mathbb{R}^n}\Vert \sigma_{z}(D_x) \psi_l(\mathcal{R}) f \Vert_{L^\infty(\mathbb{R}^n)}\lesssim  \Vert \sigma\Vert_{[\frac{n}{2}]+1, S(m^{-n\varepsilon},g) }\Vert \psi_l(\mathcal{R})f\Vert_{L^\infty(\mathbb{R}^n)}. 
\end{align*} So, we have
\begin{equation}
\Vert \sigma_{x}(D_x)f(x)\Vert_{\Lambda^s} \lesssim \Vert \sigma\Vert_{[\frac{n}{2}]+1, S(m^{-n\varepsilon},g) }\Vert f\Vert_{\Lambda^s(\mathbb{R}^n)}. 
\end{equation} Thus, we finish the proof.
\end{proof}

By following the notation in Delgado \cite{Profe2}, we extend the previous result when $0\leq \beta\leq n\varepsilon_0.$

\begin{theorem} \label{MainT'''}
Let us assume $0\leq \beta \leq n\varepsilon_0$ and let $\sigma\in S(m^{-\beta},g).$  Then 
the pseudo-differential operator $\sigma(x,D_x)$ extends to a bounded operator from  $\Lambda^s(\mathbb{R}^n)$ into  $\Lambda^{s-(n\varepsilon_0-\beta)}(\mathbb{R}^n)$  for all $0<s<1$ provided that $0<s-(n\varepsilon_0-\beta)<1.$ 
\end{theorem}
\begin{proof}
Let us consider the following factorization for $\sigma(x,D_x):$
\begin{eqnarray}
\sigma(x,D_x)=m(x,D)^{\gamma}m(x,D)^{-\gamma}\sigma(x,D_x), \,\,\gamma:=n\varepsilon_0-\beta.
\end{eqnarray} Taking into account that  $m(x,D)^{-\gamma}$ has symbol in $S(m^{-\gamma},g),$  the symbol of the operator $m(x,D)^{-\gamma}\sigma(x,D_x)$ belongs to $S(m^{-n\varepsilon_0},g).$ Consequently, we deduce the boundedness of $m(x,D)^{-\gamma}\sigma(x,D_x)$ on $\Lambda^s(\mathbb{R}^n)$ and from the continuity of $m(x,D)^\gamma$ from  $\Lambda^s(\mathbb{R}^n)$ into  $\Lambda^{s-\gamma}(\mathbb{R}^n)$  we finish the proof of the theorem.
\end{proof}

Let us observe that for all $-\infty<s<\infty$ the Besov space $B^s_{\infty,\infty}(\mathbb{R}^n)$ is defined by the norm.
\begin{equation}\Vert f\Vert_{B^s_{\infty,\infty}(\mathbb{R}^n)}:=\sup_{l\geq 0}2^{ls}\Vert \psi_l(\mathcal{R})f \Vert_{L^\infty(\mathbb{R}^n)}.
\end{equation}

Clearly, a look to the proof of our results above gives the following result.
 \begin{corollary}
 Let us assume $\varepsilon_0\leq \varepsilon<\frac{Q_0}{2n}$ and let $\sigma\in S(m^{-n\varepsilon},g).$  Then 
the pseudo-differential operator $\sigma(x,D_x)$ extends to a bounded operator on $B^s_{\infty,\infty}(\mathbb{R}^n)$ for all $-\infty<s<\infty.$ Moreover, there exists $\ell$ such that \begin{equation}
\Vert \sigma(x,D_x)f\Vert_{B^s_{\infty,\infty}(\mathbb{R}^n) }\leq C \Vert \sigma\Vert_{\ell, S(m^{-n\varepsilon},g) }\Vert f \Vert_{B^s_{\infty,\infty}(\mathbb{R}^n)}.
\end{equation} Moreover,  for $0\leq \beta \leq n\varepsilon_0$ and $\sigma\in S(m^{-\beta},g),$
the pseudo-differential operator $\sigma(x,D_x)$ extends to a bounded operator from  $B^s_{\infty,\infty}(\mathbb{R}^n)$ into  $B^{s-(n\varepsilon_0-\beta)}(\mathbb{R}^n)$  for all $-\infty<s<\infty.$ 
 \end{corollary}

 \subsection{Ruzhansky-Turunen classes in H\"older spaces} 
In this section we prove our H\"older estimate for periodic (toroidal) pseudo-differential operators. Our starting point is the following lemma which is slight variation of one due to J. Delgado (see Lemma 3.6 of \cite{Profe2}) and whose proof is only a repetition of Delgado's proof.

\begin{lemma}\label{lemma} Let $0\leq \varepsilon <1$ and $k:=[\frac{n}{2}]+1,$  let $\sigma:\mathbb{T}^n\times \mathbb{Z}^n\rightarrow \mathbb{C}$ be a symbol such that  $|\Delta_{\xi}^{\alpha}\sigma(x,\xi)|\leq C_{\alpha}\langle \xi \rangle^{-\frac{n}{2}\varepsilon-(1-\varepsilon)|\alpha|},$  for $|\alpha| \leq k.$ Let us assume that $\sigma$ is supported in $\{\xi:|\xi|\leq 1\}$ or $\{\xi:R\leq |\xi|\leq 2R\}$ for some $R>0.$ Then  $a(x,D):L^{\infty}(\mathbb{T}^n)\rightarrow L^{\infty}(\mathbb{T}^n)$ extends to a bounded linear operator with norm operator independent of $R.$ Moreover,\begin{equation}
\Vert \sigma(x,D_x)\Vert_{\mathscr{B}(L^{\infty})}\leq C  \sup\{C_{\alpha}: {|\alpha|\leq k}\}.\end{equation}
\end{lemma} 

\begin{theorem}\label{MainT*} Let $0\leq \varepsilon <1$ and $k:=[\frac{n}{2}]+1,$  let $\sigma:\mathbb{T}^n\times \mathbb{Z}^n\rightarrow \mathbb{C}$ be a symbol such that  $|\Delta_{\xi}^{\alpha}\sigma(x,\xi)|\leq C_{\alpha}\langle \xi \rangle^{-\frac{n}{2}\varepsilon-(1-\varepsilon)|\alpha|},$  for $|\alpha| \leq k.$ Then  $a(x,D):\Lambda^s(\mathbb{T}^n)\rightarrow\Lambda^s(\mathbb{T}^n)$ extends to a bounded linear operator for all $0<s<1.$ Moreover,\begin{equation}
\Vert \sigma(x,D_x)\Vert_{\mathscr{B}(\Lambda^s)}\leq C  \sup\{C_{\alpha}: {|\alpha|\leq k}\}.\end{equation}
\end{theorem}

\begin{proof}[Proof of Theorem \ref{MainT*}] Our proof consists of two steps. In the first one, we prove the statement of the theorem for periodic Fourier multipliers, i.e., toroidal pseudo-differential operators depending only on the Fourier variable $\xi.$ Later, in the second step, we extend the result for general periodic  operators.\\

\noindent{{\textit{Step 1.}}} Let us consider the operator $\mathcal{R}:=(I-\frac{1}{4\pi^2}\mathcal{L}_{\mathbb{T}^n})^{\frac{1}{2}},$ where $\mathcal{L}_{\mathbb{T}^n}$ is the Laplacian on the torus $\mathbb{T}^n,$ and let us fix a dyadic
decomposition of its spectrum:  we choose a function $\psi_0\in C^{\infty}_{0}(\mathbb{R}),$  $\psi_0(\lambda)=1,$  if $|\lambda|\leq 1,$ and $\psi(\lambda)=0,$ for $|\lambda|\geq 2.$ For every $j\geq 1,$ let us define $\psi_{j}(\lambda)=\psi_{0}(2^{-j}\lambda)-\psi_{0}(2^{-j+1}\lambda).$ Then we have
\begin{eqnarray}\label{deco1}
\sum_{l\in\mathbb{N}_{0}}\psi_{l}(\lambda)=1,\,\,\, \text{for every}\,\,\, \lambda>0.
\end{eqnarray}
We will use the following characterization for H\"older spaces in terms of dyadic decompositions (see Stein \cite{Eli}): $f\in \Lambda^s(\mathbb{T}^n),$ if  and only if, 
\begin{equation}
\Vert f \Vert_{\Lambda^s(\mathbb{T}^n)}\asymp \Vert f\Vert_{B^s_{\infty,\infty}(\mathbb{T}^n)}:=\sup_{l\geq 0}2^{ls}\Vert \psi_l(\mathcal{R})f \Vert_{L^\infty(\mathbb{T}^n)}
\end{equation}
where $\psi_l(\mathcal{R})$ is defined by the functional calculus associated to the self-adjoint operator $\mathcal{R}.$ If $\sigma(D_x)=\sigma(x,D_x  )$ has a symbol depending only on the Fourier variable, then 
\begin{equation}
\Vert \sigma(D_x) f \Vert_{\Lambda^s(\mathbb{T}^n)}\asymp \Vert    \sigma(D_x)f\Vert_{B^s_{\infty,\infty}(\mathbb{T}^n)}:=\sup_{l\geq 0}2^{ls}\Vert \psi_l(\mathcal{R})\sigma(D_x)f \Vert_{L^\infty(\mathbb{T}^n)}.
\end{equation}
Taking into account that the operator $\sigma(D_x)$ commutes with $\psi_l(\mathcal{R})$ for every $l,$ that 
\begin{equation}
\psi_l(\mathcal{R})\sigma(D_x)= \sigma(D_x)\psi_l(\mathcal{R})= \sigma_{l}(D_x)\psi_l(\mathcal{R})
\end{equation}
where $\sigma_l(D_x)$ is the pseudo-differential operator with symbol $$\sigma_{l}(\xi)=\sigma(\xi) \cdot 1_{\{\xi:2^{l-1}\langle \xi\rangle\leq 2^{l+1}   \}},$$
and that $\sigma_{l}(D_x)$ has a symbol supported in $\{\xi:2^{l-1}\langle \xi\rangle\leq 2^{l+1}   \},$ by Lemma \ref{lemma} we deduce that $\sigma_{l}(D_x)$ is a bounded operator on $L^{\infty}(\mathbb{T}^n)$ with operator norm independent on $l.$ In fact,  $\sigma_{l}$ satisfies the symbol inequalities
$$|\Delta_{\xi}^{\alpha}\sigma_l(\xi)|\leq C_{\alpha}\langle \xi \rangle^{-\frac{n}{2}\varepsilon-(1-\varepsilon)|\alpha|},$$ for all $|\alpha| \leq k,$ and consequently
\begin{equation}
\Vert \sigma_{l}(D_x)\Vert_{\mathscr{B}(L^{\infty})}\leq C \sup\{C_{\alpha}: {|\alpha|\leq k}\}.\end{equation} So, we have
\begin{align*}
\Vert \psi_l(\mathcal{R})\sigma(D_x) & f \Vert_{L^\infty(\mathbb{T}^n)} \\
&=\Vert   \sigma_l(D_x) \psi_l(\mathcal{R})f \Vert_{L^\infty(\mathbb{T}^n)}\leq \Vert \sigma_{l}(D_x)\Vert_{\mathscr{B}(L^{\infty})} \Vert    \psi_l(\mathcal{R})f \Vert_{L^\infty(\mathbb{T}^n)} \\
& \lesssim \sup\{C_{\alpha}: {|\alpha|\leq k}\}\Vert    \psi_l(\mathcal{R})f \Vert_{L^\infty(\mathbb{T}^n)}.
\end{align*}
As a consequence, we obtain
\begin{align*}\Vert \sigma(D_x) f \Vert_{\Lambda^s(\mathbb{T}^n)} &\asymp \sup_{l\geq 0}2^{ls}\Vert \psi_l(\mathcal{R})\sigma(D_x)f \Vert_{L^\infty(\mathbb{T}^n)}\\
& \lesssim \sup\{C_{\alpha}: {|\alpha|\leq k}\}\sup_{l\geq 0} 2^{ls} \Vert    \psi_l(\mathcal{R})f \Vert_{L^\infty(\mathbb{T}^n)}\\
&\asymp \sup\{C_{\alpha}: {|\alpha|\leq k}\}\Vert f\Vert_{\Lambda^s(\mathbb{T}^n)}. 
\end{align*}
\noindent{{\textit{Step 2.}}} Now, we extend the H\"older boundedness from multipliers to pseudo-differential operators by using a technique developed by Ruzhansky, Turunen and  Wirth (see Ruzhansky and Turunen \cite{Ruz} and Ruzhansky and Wirth \cite{Ruz3}). So,  let us define for every $z\in\mathbb{T}^n,$ the multiplier
\begin{center}
$ \sigma_z(D_{x})f(x)=\sum_{\eta\in \mathbb{Z}^n}e^{i2\pi \langle x, \eta\rangle}\sigma(z,\eta)\widehat{f}(\eta).$
\end{center}
For every $x\in\mathbb{T}^n$ we have the equality,
\begin{center}
$ \sigma_{x}(D_{x})f(x)=\sigma(x,D_x)f(x),$
\end{center}
and we can estimate the H\"older norm of the function $\sigma(x,D_x)f,$ as follows
\begin{align*}
\Vert \sigma_{x}(D_x)f(x)\Vert_{\Lambda^s} &\asymp \sup_{l\geq 0}2^{ls} \textnormal{esssup}_{x\in\mathbb{T}^n}|\psi_l(\mathcal{R})\sigma_x(D_x)f(x) |\\
&\leq  \sup_{l\geq 0}2^{ls} \textnormal{esssup}_{x\in\mathbb{T}^n}   \sup_{z\in\mathbb{T}^n}|\psi_l(\mathcal{R})\sigma_z(D_x)f(x) | \\
&=  \sup_{l\geq 0}2^{ls} \textnormal{esssup}_{x\in\mathbb{T}^n}   \sup_{z\in\mathbb{T}^n}| \sigma_z(D_x)  \psi_l(\mathcal{R})f(x) | \\
&\leq \sup_{l\geq 0}2^{ls}  \textnormal{esssup}_{x\in\mathbb{T}^n}   \sup_{z\in\mathbb{T}^n} \textnormal{ess sup}_{\varkappa\in\mathbb{T}^n} |\sigma_z(D_\varkappa)  \psi_l(\mathcal{R})f(\varkappa) | \\
&= \sup_{l\geq 0}2^{ls}   \sup_{z\in\mathbb{T}^n} \Vert \sigma_z(D_\varkappa)  \psi_l(\mathcal{R})f(\varkappa) \Vert_{L^\infty(\mathbb{T}^n)} .\\
\end{align*} From the estimate for the operator norm of multipliers proved in the first step, we deduce
\begin{align*}\sup_{z\in\mathbb{T}^n}\Vert \sigma_{z}(D_x) \psi_l(\mathcal{R}) f \Vert_{L^\infty(\mathbb{T}^n)}\lesssim  \sup\{C_{\alpha}: {|\alpha|\leq k}\}\Vert \psi_l(\mathcal{R})f\Vert_{L^\infty(\mathbb{T}^n)}. 
\end{align*} So, we have
\begin{equation}
\Vert \sigma_{x}(D_x)f(x)\Vert_{\Lambda^s} \lesssim \sup\{C_{\alpha}: {|\alpha|\leq k}\}\Vert f\Vert_{\Lambda^s(\mathbb{T}^n)}. 
\end{equation} Thus, we finish the proof.
\end{proof}

\begin{corollary}\label{MainT''''}  Let $0<\rho\leq 1,$ $0\leq \delta\leq 1,$ $\ell\in\mathbb{N},$ $k:=[\frac{n}{2}]+1,$  and let $A:C^{\infty}(\mathbb{T}^n)\rightarrow C^{\infty}(\mathbb{T}^n)$ be a pseudo-differential operator with symbol $\sigma$ satisfying  
\begin{equation}\label{eqMaint'''''''}
\vert \partial_x^\beta\Delta_{\xi}^{\alpha}\sigma(x,\xi)\vert\leq C_{\alpha}\langle \xi \rangle^{-m-\rho|\alpha|+\delta|\beta|}
\end{equation}
 for all $|\alpha| \leq k,$ $|\beta|\leq \ell.$ Then  $A:\Lambda^s(\mathbb{T}^n)\rightarrow \Lambda^s(\mathbb{T}^n)$ extends to a bounded linear operator for all $0<s<1$ provided that $m\geq \delta \ell+\frac{n}{2}(1-\rho).$
\end{corollary}

\begin{proof}
Let us  observe that 
$
\langle \xi \rangle^{-m-\rho|\alpha|+\delta|\beta|}\leq \langle \xi \rangle^{-\frac{n}{2}(1-\rho)+\rho|\alpha|},
$ when $m\geq \delta \ell+\frac{n}{2}(1-\rho).$ So, we finish the proof if we apply Theorem \ref{MainT}.
\end{proof}
\begin{remark}
In the proofs of our H\"older estimates we have used the equivalence between the Besov space $B^{s}_{\infty,\infty}(\mathbb{T}^n)$ and the H\"older space $\Lambda^s(\mathbb{T}^n),$ for $0<s<1.$ Consequently,  hypothesis of Theorem \ref{MainT} and Corollary \ref{MainT''''} also assure the boundedness of $\sigma(x,D_x)$ on $B^{s}_{\infty,\infty}(\mathbb{T}^n)$ for all $-\infty<s<\infty.$
\end{remark} 
Besov boundedness results for  operators in Ruzhansky-Turunen classes for  $1<p<\infty$ and $0<q<\infty$ can be found in Cardona \cite{CardBesovCras}, Cardona and Ruzhansky \cite{CardonaRuzhansky2019} and references therein.
 \subsection*{Acknowledgements}  The author wants to express his gratitude to the referee
who pointed out a number of suggestions helping to improve the presentation of the
manuscript.

\bibliographystyle{amsplain}

\end{document}